\documentclass[reqno,12pt]{amsart} %leqno is the option to put formula numbers on the left side
\usepackage{amscd,amssymb,pstricks,latexsym,amsbsy,xypic,mathrsfs,verbatim}
\usepackage[arrow,matrix,cmtip,curve]{xy}
\usepackage{lscape}
\usepackage{enumerate}
\usepackage{graphicx}

\makeatletter
\@namedef{subjclassname@2010}{%
  \textup{2010} Mathematics Subject Classification}

\numberwithin{equation}{section}

\theoremstyle{plain}
\newtheorem{prop}{Proposition}[section]
\newtheorem{thm}[prop]{Theorem}
\newtheorem{cor}[prop]{Corollary}
\newtheorem{lem}[prop]{Lemma}

\theoremstyle{definition}
\newtheorem{example}[prop]{Example}

\newcommand{\scr}[1]{\mathscr #1}

\def\lra{\longrightarrow}

\def\bbn{{\mathbb N}}

\def\dim{{{\rm dim}\,}}
\def\bfdim{{{\rm\bf dim}\,}}

\def\End{{{\rm End}\,}} \def\thz{\theta}
\def\Hom{{{\rm Hom}\,}}

\def\Ext{{{\rm Ext}\,}}

\def\mod{{\text{\rm mod}}}

  \def\lra{\longrightarrow}

\def\cal{\mathcal}

\def\dim{{\rm dim}\,}

\begin{document}

\title[{On the slopes of semistable representations of tame quivers}]{On the slopes of semistable representations\\ of tame quivers}

\author[X. Wang]{Xintian Wang}

\address{School of Mathematical Sciences, Beijing Normal University,
Beijing 100875,  China.}
\email{wangxintian0916@126.com }

\thanks{Supported by the Natural Science Foundation of China.}

%\date{\today}



\begin{abstract} Stability conditions play an important role in the study of
representations of a quiver. In the present paper, we study semistable representations
of quivers. In particular, we describe the slopes of semistable representations of a
tame quiver for a fixed stability condition.
\end{abstract}

\maketitle

\section{Introduction and preliminaries}

The notion of stability was firstly introduced by Mumford in his
work on the geometric invariant theory in 1960s and soon became
widely used as a technical tool while constructing moduli varieties.
In \cite{ADK} King set up the semistability and stability in the
language of the module category over a finite dimensional algebra,
more generally for an arbitrary abelian category.

Let $Q=(Q_0,Q_1)$ be a finite acyclic quiver (i.e., without oriented
cycles) with vertex set $I=Q_0$ and arrow set $Q_1$. Let $k$ be an
algebraically closed field and $\mod kQ$ denote the category of
finite dimensional modules over the path algebra $kQ$ (equivalently,
finite dimensional representations of $Q$ over $k$). Following Reineke
\cite{MR}, a stability in $\mod kQ$ is defined relative to a slope
function $\mu$ on $\bbn I\backslash\{0\}$. More precisely, a
$kQ$-module $X$ is called semistable (resp. stable) if $\mu(\bfdim
U)\leq \mu(\bfdim X)$ (resp. $\mu(\textbf{dim} U)< \mu(\textbf{dim}
X)$) for all proper submodules $0\neq U\subseteq X$, where $\bfdim X$
and $\bfdim U$ denote the dimension vectors of $X$ and $U$,
respectively. In this case, $\mu(\bfdim X)$ is called the slope of
$X$.

For each $a\in\mathbb Q$, let $\mod^{a}kQ$ denote the full subcategory of $\mod kQ$ consisting
of semistable $kQ$-modules of slope $a$. It is known that each
$\mod^{a}kQ$ is an abelian category of $\mod kQ$. In case $Q$ is a Dynkin or tame quiver, the subcategory
$\mod^{a}kQ$ has been characterized in \cite{IPT,IT}.

The main purpose of the present paper is to describe the slopes of
semistable modules of $kQ$ when $Q$ is a tame quiver. This is based
on an investigation of the structure of the subcategories $\mod^{a} kQ$.

\medskip

In the following we briefly review some basic facts about finite dimensional algebras
and their representations. We also introduce the stability condition for a finite dimensional
algebra. We refer to \cite{ASS,ARS,HP,DDPW} for more details and complete treatments.

Let $k$ be a field and $A$ be a finite dimensional
algebra over $k$. By $\mod A$ we denote the category of all finite
dimensional left $A$-modules. Let $I$ denote the set
of isoclasses of simple objects in $\mod A$, and fix a set
$\{S_i\mid i\in I\}$ of representatives of the isoclasses in $I$.
For any $M\in$ $\mod A$, let $[M]$ denote the isoclass of $M$ and $\bfdim M$
the dimension vector of $M$. More precisely, if $\bfdim M=(x_i)_{i\in I}$, then
$x_i$ is the number of composition factors isomorphic to $S_i$ in a composition
series of $M$. Further, set $f_i=\dim_{k}\End_{A}(S_i)$.

From now onwards, we always assume that $A$ is hereditary. The Euler form of $A$ is defined by
$$\langle\bfdim M,\bfdim N\rangle=\dim_{k}\Hom_{A}(M,N)-\dim_{k}\Ext_{A}^{1}(M,N),$$
 where $M,N\in\mod A$. Further, let $\Gamma_{A}$ be the Auslander--Reiten quiver of $A$
 with the Auslander--Reiten translation $\tau=\tau_A$. A
connected component $\mathcal{P}$ in $\Gamma_{A}$ is called
preprojective (resp. preinjective) if, for each vertex $[M]$ in
$\mathcal{P}$, the supremum (resp. infimum) of the lengths of the
paths ending (resp. starting) at $[M]$ is finite. Otherwise, it is
called regular.

 An indecomposable
$A$-module is called preprojective (resp. preinjective) if it
belongs to a preprojective (resp. preinjective) component of $\Gamma
_{A}$ and an arbitrary $A$-module is called preprojective (resp.
preinjective) if it is a direct sum of indecomposable
preprojective (resp. preinjective) modules. Otherwise, it is a regular module.

Suppose now that  $A$ is of tame type. Let $\delta$ be the minimal positive imaginary
root of $A$. Recall from \cite{DR} that the defect $\partial(M)$ of a module $M$ is defined to be the
integer $\langle \delta, \bfdim M \rangle$. Then an
indecomposable module M is preprojective (resp. regular,
preinjective) if and only if $\partial(M)<0$ (resp.
$\partial(M)=0$, $\partial(M)>0$ )

A translation quiver ($\mathcal{T},\tau$) is defined to be a stable
tube of rank $r\geq1$ if there is an isomorphism of translation
quivers $\mathcal{T} \cong \mathbb{Z}\mathbb{A}_{\infty}$$/(\tau
^r)$. A stable tube of rank r=1 is defined to be a homogeneous tube;
otherwise, it is a non-homogeneous tube. A representation in a
stable tube which has only one arrow to and from it is called
quasi-simple.

The following result is well known; see \cite{DR}.

\begin{lem} Let $A$ be a finite dimensional hereditary algebra of tame type. Then
the Auslander--Reiten quiver $\Gamma_{A}$ of $A$ contains a
preprojective component $\mathcal{P}$, a preinjective component
$\mathcal{I}$, and a $\mathbb{P}$$^1(k)$-family
$\{\mathcal{T}_\lambda\}$ of stable tubes of which only finitely
many ones are non-homogeneous. Moreover, for $\lambda, \lambda'\in
\mathbb{P}$$^1(k)$, we have
\begin{itemize}

\item[(1)] $\Hom({\mathcal I},{\mathcal P} \bigcup {\mathcal T}_\lambda)=0$ and
$\Hom({\mathcal T}_\lambda,{\mathcal P})=0$,

\item[(2)] $\Hom({\mathcal T}_{\lambda},{\mathcal T}_{\lambda'})=0$ if
$\lambda \neq \lambda'$.
\end{itemize}

\end{lem}

Take $\theta=(\theta_i)_{i\in I}\in{\mathbb Z} I$ and define a linear form on
$\mathbb{Z}$$I$ by setting $\theta (d)= \sum_{i\in I}\theta_id_if_i$
(for simplicity, we still denote by $\theta$ the linear form), where $d=(d_i)_{i\in I}\in{\mathbb Z} I$.
We call $\theta$ a weight for $A$. The slope function $\mu$ on
$\mathbb{N}$$I\backslash \{0\}$ associated to $\thz$ is defined by
$$\mu(d)=\frac{\sum\limits_{i\in I}\theta_id_if_i }{ \sum\limits_{i\in
I}d_if_i}.$$
For each $ M\in\mod A$, write $\mu(M$) for $\mu(\bfdim M)$.
A module $M\in\mod A$ is called {\it semistable} (resp. {\it
stable}) if $\mu(U)\leq \mu(M)$ (resp. $\mu(U)< \mu(M)$) for
all proper submodules $0\neq U\subset M$.

\medskip

The following two lemmas are well known, see, for example,
\cite{MR}.

\begin{lem}\label{ses} Given a short exact sequence
$$0\lra M\lra X\lra N\lra 0.$$ in  $\mod A$, we have
$$\mu(M)\leq \mu(X) \Longleftrightarrow\mu(X)\leq \mu(N)\Longleftrightarrow\mu(M)\leq
\mu(N)\;\text{ and}$$
$${\rm min}(\mu(M),\mu(N))\leq \mu(X)\leq{\rm max}(\mu(M),\mu(N)).$$
If $\mu(M)=\mu(X)=\mu(N)$, then
$X$ is semistable if and only if $M$ and $N$ are semistable.
\end{lem}

By the above lemma, if $M=M_1\oplus M_2$ with both $M_1$ and $M_2$
indecomposable, then $M$ is semistable if and only if $M_1$ and
$M_2$ are semistable and $\mu(M_1)=\mu(M_2)=\mu(M)$. Hence, for
semistable $A$-modules, it is enough to consider the indecomposable
ones.

\medskip

\begin{lem}\label{equivalent definition} Let $X\in$ $\mod A$. Then $X$ is
 semistable {\rm(}resp. stable{\rm)} if and only if
$\mu(X)\leq \mu(U)$ {\rm(}resp. $\mu(X) < \mu(U)${\rm)} for all
proper quotient modules $U$.
\end{lem}
\medskip

For each $a\in \mathbb{Q}$, denote by $\mod^{a}A$ the full
subcategory of $\mod A$ consisting of semistable $A$-modules of
slope $a$. By convention, we always assume that $\mod^{a}A$ consists of the zero module $0$.

%\medskip
\begin{lem}\label{prop-of-sub}
For each $a\in \mathbb{Q}$,           the category $\mod^{a}A$ is an abelian
subcategory of $\mod A$ whose simple objects are the  indecomposable
stable $A$-modules of slope $a$. Moreover, we have that
$\Hom(\mod^aA,\mod^bA)=0$ whenever $a> b$.
\end{lem}

%%%3333333333333333333333333333333333333333333333333333333333333333333333333333333333333333333333333333333333333333333%

\section{Category of semistable $kQ$-modules of slope $a$ }

In this section, we recall some results from \cite{IPT,IT} which
will be needed in the next section in order to prove our main result.

In the following, we assume that $k$ is an algebraically closed field, $Q$ is an acyclic quiver,
and $A$ is the path algebra $kQ$. Thus, the set $I$ of isoclasses of simple $A$-modules
is identified with the vertex set $Q_0$, and $f_i=\dim_{k}\End_{A}(S_i)=1$ for all $i\in I$.
In case $Q$ is a tame quiver, we denote by $\delta$
the minimal positive imaginary root of $Q$, and let $\mathcal{P}$ and
$\mathcal{I}$ be the preprojective and preinjective components,
respectively, and let $\mathcal{R}$ be the union of all tubes of the Auslander--Reiten quiver $\Gamma_A$.

\medskip

We first introduce a different stability notion as follows
\cite{HP}. Let $\theta=(\theta_i)_{i\in I}$ be a weight for $Q$. We
denote $\theta(\rm{\bf dim}M)$ by $\theta(M)$. A module $M\in$ $\mod
A$ is called $\theta$-semistable (resp. stable) if $\theta(M)=0$ and
$\theta(U)\leq 0$ (resp. $\theta(U)< 0$) for any proper submodule
$0\neq U\subseteq M$. Finally, by $\mod_{\theta}A$ we denote the
full subcategory of $\mod A$ consisting of all the
$\theta$-semistable modules.

\begin{lem} \label{stab-cond-equiv} Let $\theta=(\theta_i)_{i\in I}$ be
a weight. Then for each $a\in \mathbb{Q}$, $\mod^{a}A$= $\mod_{\theta '}A$, where
$\theta'=\theta-a\theta_0$ and $\theta_0= (1)_{i\in I}$.
Conversely, for a weight $\omega=(\omega_i)_{i\in I}$, there exists
a weight $\theta=(\theta_i)_{i\in I}$ and $a\in \mathbb{Q}$ such
that $\mod_{\omega}A= $ $\mod^{a}A$.
\end{lem}

\begin{proof} By the definition, for a weight $\theta=(\theta_i)_{i\in I}$, if $\mu$ is the
slope function associated with $\thz$, then $\mu(M)= a$ if and only if $(\theta - a\theta_0)(M)=0$.
Moreover, $\mu(M)\leq a$ if and only if $(\theta - a\theta_0)(M)\leq 0$. This implies the
desired statements.
\end{proof}

By \cite{IPT}, we have the following statement.

\begin{thm}\label{subcat-equiv}  Let $Q$ be a tame quiver and $\theta=(\theta_i)_{i\in I}$ be
a weight for $A=kQ$. Then for each $a\in\mathbb Q$, the subcategory $\mod^{a}A$ is
equivalent to one of the following two categories:
\begin{itemize}
\item[(1)] the module category $\mod kQ'$ of the path algebra $kQ'$ for a Dynkin or tame quiver $Q'$;

\item[(2)] the full subcategory $\mathcal{R'}$ consisting of all the regular objects of $\mod kQ'$ with
$Q'$ a possibly disconnected tame quiver (i.e., a quiver with one
tame component and all other components (if any) Dynkin).
\end{itemize}
\end{thm}

The following statement is an easy consequence of the above theorem.

\begin{cor}\label{Euler-type-equiv}  Let $Q$ and $Q'$ be as in Theorem {\rm\ref{subcat-equiv}}
with $A=kQ$ and $B=kQ'$. Let $\langle-,-\rangle_{B}$ and
$\langle-,-\rangle_{A}$ be the Euler forms associated to $B$ and $A$,
respectively. Then for any $M,N \in$ $\mod B$, $M,N$ can be viewed as
$A$-modules and, moreover,
$$\langle M,N\rangle_{B}= \langle M,N\rangle_A.$$
\end{cor}

\begin{prop}\label{decr} Let $Q$ be a tame quiver and $\theta=(\theta_i)_{i\in I}$ be
a weight for $A=kQ$. If there is an indecomposable $A$-module of dimension vector
$m\delta$ in $\mod^{a}A$ for some $m\geq 1$, then there is an indecomposable
$A$-module of dimension vector $\delta$ in $\mod^{a}A$.
\end{prop}

\begin{proof} Let $M$ be an indecomposable $A$-module in $\mod^{a}A$ with
$\bfdim M= m\delta$ for some $m\geq 1$. Then $M$ has a submodule $N$ with
$\bfdim N=\delta$ which gives an exact sequence
$$0 \longrightarrow N \longrightarrow M \longrightarrow M/N \longrightarrow 0.$$
Since $\mu(N)=\mu(\delta)=\mu(M)$, we have by Lemma \ref{prop-of-sub}
that $N$ is semistable. Hence, $N$ lies in $\mod^{a}A$.
\end{proof}

\begin{cor}\label{classification of slope} We keep the notations as in the above proposition and take $a\in\mathbb Q$.

{\rm (1)} If $a\neq \mu(\delta)$, then $\mod^{a}A$ is equivalent to
$\mod kQ'$, where $Q'$ is a Dynkin quiver.

{\rm (2)} If $\mod^aA \cap \mathcal{P} \neq \varnothing$ or
$\mod^aA \cap \mathcal{I} \neq \varnothing$, then $\mod^{a}A$ is
equivalent to $\mod kQ'$, where $Q'$ is a Dynkin or tame
quiver. Moreover, if $\mod^{a}A$ contains an indecomposable module with
dimension vector $\delta$, then $Q'$ is a tame quiver.
Otherwise, $Q'$ is a Dynkin quiver.

{\rm (3)} If $\mod^aA \subset \mathcal{R}$, and has no indecomposable
object with dimension $\delta$, then ${\rm mod}^aA$ is equivalent to
$\mod kQ'$, where $Q'$ is a Dynkin quiver.

{\rm (4)} If $\mod^aA \subset \mathcal{R}$, and has an indecomposable
object with dimension $\delta$, then $\mod^{a}A$ is equivalent to
the full subcategory consisting of all the regular objects of $\mod
kQ'$, where $Q'$ is a possibly disconnected tame quiver.
\end{cor}

\begin{proof} (1) Suppose that $\mod^{a}A$ is equivalent to $\mod kQ'$, where $Q'$
is a tame quiver. By Corollary \ref{Euler-type-equiv}, there is an
indecomposable module $M$ in $\mod^{a}A$ with $\bfdim M= m\delta$
for some $m\geq 1$. Then $a= \mu(M)=\mu(m\delta)=\mu(\delta)$, which
is a contradiction. Similarly,  $\mod^{a}A$ is not equivalent to the
full subcategory consisting of all the regular objects of $\mod
kQ'$, where $Q'$ is a possibly disconnected tame quiver. The
statement follows from Theorem \ref{subcat-equiv}.

(2) For the first statement, suppose that $\mod^{a}A$
is equivalent to $\mathcal{R'}$, where $\mathcal{R'}$ denotes the
full subcategory consisting of all the regular objects of $\mod kQ'$
with $Q'$ a possibly disconnected tame quiver. Let $\delta '$
be the minimal positive imaginary root of $Q'$ and take $M \in
\mathcal{R'}$ with $\textbf{dim} M= \delta '$. Then for all $N$ in
$\mod^{a}A$,
$$\langle M,N\rangle_{A}=\langle M,N\rangle_{B}=0.$$ Here
$M,N$ are viewed as both $A$-modules and $B$-modules. But since
$\mod^aA \cap \mathcal{P} \neq \varnothing$ or $\mod^aA \cap
\mathcal{I} \neq \varnothing$, there exists a module $N_0 \in$
$\mod^{a}A$ such that $\langle M,N_0\rangle_{A}\neq 0 $, a
contradiction. The second statement follows from Corollary
\ref{Euler-type-equiv}.

(3) By Proposition \ref{decr}, there exists no indecomposable module
in $\mod^{a}A$ with dimension vector $m\delta,m\geq 1$. Then the
statement follows from Corollary \ref{Euler-type-equiv} and Theorem
\ref{subcat-equiv}.

(4) The proof is similar to (2).

\end{proof}

\section{The slopes of semistable $kQ$-modules}

In this section, we describe the slopes of semistable $kQ$-modules
in case $Q$ is a tame quiver. The main result is stated in Theorem
\ref{main theorem}. We keep all the notations in the previous
section. In particular, $Q=(I=Q_0,Q_1)$ denotes an acyclic quiver
and $A=kQ$ is the path algebra of $Q$ over an algebraically closed
field.

We denote by $\mod ^{ss}A$ the full subcategory of $\mod A$
consisting of semistable $A$-modules. Hence,
$\mod^{ss}A=\displaystyle\cup_{a\in \mathbb{Q}} \mod ^aA$. For a weight $\theta=(\theta_i)_{i\in I}$ for $A$,
define
$${\scr X}_\theta=\{ a\in {\mathbb{Q}} \mid \mod^aA \;\text{ is non-zero} \}.$$
 Our main aim in this section is to describe the set ${\scr X}_\theta$ in case $Q$ is a tame
quiver. First of all, we have the following facts in some special cases.

\begin{prop} {\rm (1)} $|{\scr X}_\theta|=1$ if and only if all $A$-modules are
semistable, i.e. $\mod ^{ss}A=\mod A$. In other words, all $\theta_i$, $i\in I$,
coincide.

{\rm (2)} $|{\scr X}_\theta|=2$ if and only if $Q$ contains two full
subquivers $Q'=(Q'_0,Q'_1)$ and $Q''=(Q_0'',Q_1'')$ such that
$Q_0=Q_0'\cup Q_0''$  and there are no arrows from $Q_0''$ to $Q_0'$
and $\theta_{i_1}=\theta_{j_1}<\theta_{i_2}=\theta_{j_2}$ for all
$i_1,j_1\in Q_0',i_2,j_2\in Q_0''$.
\end{prop}

\begin{proof} {\rm (1)} Since each simple module $S_i$ is semistable, it follows that $\theta _i=\mu(S_i)\in
{\scr X}_\theta$. Thus, if $|{\scr X}_\theta|=1$, then
$\theta_i=\theta_j$ for all $i\neq j\in Q_0$. This implies that
 the slopes of all $A$-modules are equal. Therefore, all $A$-modules are semistable.

Conversely, assume that all $A$-modules are semistable. Suppose
$|{\scr X}_\theta|>1$, i.e., there are $i,j\in Q_0$ such that
$\theta _i\neq \theta_j$. Without loss of generality, we assume
$\theta _i< \theta_j$. This implies that the semisimple module
$S_i\oplus S_j$ is not semistable. This is a contradiction. Hence,
$|{\scr X}_\theta|=1$.

{\rm (2)} Suppose $|{\scr X}_\theta|=2$, say ${\scr X}_\theta=\{a,b\}$ with
$a<b$. Set
$$Q_0'=\{i\in I\mid \theta_i=a\}\;\text{ and }\;Q_0''=\{i\in I\mid \theta_i=b\}.$$
 Let $Q'$ and $Q''$ be the full subquivers of $Q$ with vertex sets $Q_0'$ and $Q_0''$,
respectively. Then $Q_0$ is the disjoint union of $Q'_0$ and
$Q''_0$. Suppose there is an arrow $j\longrightarrow i$ with $i\in
Q_0'$ and $j\in Q_0''$. Consider the indecomposable $A$-module $M$
with socle $S_i$ and $M/S_i\cong S_j$. Then $M$ is semistable and
$\mu(\bfdim M)=(a+b)/2\in{\scr X}_\theta$. This contradicts the
assumption ${\scr X}_\theta=\{a,b\}$ since $a<(a+b)/2<b$. Therefore, there are no arrows
from $Q_0''$ to $Q_0'$.

The converse follows from the fact that each semistable
$A$-module is either a $kQ'$-module or a $kQ''$-module.
\end{proof}

\medskip
From now onwards, we assume that $Q$ is a (connected) tame quiver
which is obtained from a (connected) Dynkin quiver $\Gamma$ of type
$A,D,E$ by adding a vertex. This gives symmetric Cartan matrices
$C_{Q}$ and $C_\Gamma$ of $Q$ and $\Gamma$, respectively. Thus, we
have the associated Kac--Moody Lie algebras ${\frak g}(C_{Q})$ and
${\frak g}(C_\Gamma)$. Let $\Delta _0$ and $\Delta_0^+$ be the set
of real roots and the set of positive real roots of ${\frak
g}(C_\Gamma)$, respectively. By \cite{K}, the set of positive real
roots of ${\frak g}(C_{Q})$ can be described as
$$\Delta _+^{\rm re}=\{\alpha+n\delta \mid \alpha\in \Delta
_0,n\geq1\}\cup \Delta_0^+,$$
 and its set of imaginary roots is $\Delta^{\rm im}={\mathbb Z}\delta\backslash\{0\}$, where $\delta$
 denotes the minimal positive imaginary root of $Q$.

 Let $\cal P$ be the preprojective component of the Auslander--Reiten quiver
 of $A=kQ$. By \cite{R}, the dimension vectors of $P\in\cal P$ are positive
 real roots of ${\frak g}(C_{Q})$. Let $P_1,P_2,\ldots,P_N$ be all the indecomposable
preprojective $A$-modules, up to isomorphism, with $\bfdim
P_i=\alpha_i<\delta$. For each $P\in \mathcal{P}$, $\bfdim
P=\alpha+n\delta$, with $\alpha\in\Delta^{\text{re}}_{+}$ and
$\alpha<\delta$. Then $0>\partial(P)=\langle\delta,
\alpha+n\delta\rangle=\langle\delta, \alpha\rangle$. Thus, the
indecomposable module $X$ with $\bfdim X=\alpha$ is preprojective.
Hence, $X\cong P_i$ and $\alpha=\alpha _i$ for some $1\leq i\leq N$,
i.e., $\bfdim P=\alpha _i+n\delta$.

The following fact is well known.

\begin{lem}\label{existence of homomorphism} Let $M\in\mathcal{P}$.
 Then there exists $m\gg 0$ such that for each projective module $P$,
 $\Hom_A(M,\tau^{-n} P)\neq0$ whenever $n\geq m$.
\end{lem}

\begin{lem}\label{maximal slope} Suppose $M\in\mathcal{P}$ is
semistable satisfying $\mu(P_i)\leq\mu(M)$ for all $1\leq i\leq N$ and
$\mu(\delta)<\mu(M)$. Then there are only finitely many semistable
modules in $\mathcal{P}$.
\end{lem}

\begin{proof}  By Lemma \ref{existence of homomorphism}, there is $m\gg0$ such that for each projective module $P$
and $m\geq n$,  $\Hom_A(M,\tau^{-m} P)\neq0$. We can assume that $\bfdim \tau^{-m} P >\delta$.
This implies that $\mu(M)>\mu(\tau^{-m} P)$. By Lemma \ref{prop-of-sub}, $\tau^{-m} P$ is not
semistable. Therefore, there are only finitely many semistable modules in
$\mathcal{P}$.
\end{proof}

Now we recall some facts about the regular $A$-modules from
\cite{CB1,CB2,R}. Let $\cal T$ be a tube of rank $r$ in the
Auslander--Reiten quiver of $A$. Let $E_1,E_2,\ldots,E_r$ be the
quasi-simple modules in $\cal T$ with $\tau(E_i)=E_{i+1},1\leq i\leq r$, where
$E_{r+1}=E_1$. Let $E_{i,j}$ denote the indecomposable module in $\cal T$
with quasi-length $j$ and quasi-socle $E_i$. It is known that
$E_{i,j}$ is regular uniserial with regular composition factors of the form
$E_i,\tau^{-1} E_i,\ldots,\tau^{-(r-1)}E_i$ and
$\textbf{dim}E_{i,j}=\bfdim E_{i,j_0}+n\delta$, where $j=j_0+nr$ for
some $0\leq j_0< r$ and $n\geq 0$. We have the following known fact.

\begin{prop}\label{r3} Let $1\leq i\leq r$ and $m\geq 1$. Then for any $1\leq s\leq r$ and $m\delta\leq j< (m+1)\delta$,
$$\Hom_A(E_{i,m\delta},E_{s,j})\neq0\;\text{ and }\;\Hom_A(E_{s,j},E_{i,(m+1)\delta})\neq0.$$
\end{prop}

\begin{lem}\label{dimension vector comparison} Let $P\in\mathcal{P}$. Assume that
$\bfdim P=\alpha_i+n\delta$ and $M$ is a submodule of P with $\bfdim
M=\alpha_j+n'\delta$, where $1\leq i,j\leq N$. Then $n'\leq n$.
\end{lem}

\begin{proof} Suppose $n'> n$. Since $\textbf{dim}M=\alpha_j+n'\delta <
\bfdim P=\alpha_i+n\delta$, it follows that $(n'-n)\delta \leq
\alpha_i - \alpha_j$. But $(n'-n)\delta$ is positive and sincere.
This is impossible. Therefore, $n'\leq n$.
\end{proof}

Now we state our main theorem.

\begin{thm}\label{main theorem} {\rm (1)} If $\mod^{\mu(\delta)}A$ is equivalent to $\mod kQ'$,
where $Q'$ is a Dynkin or tame quiver, then ${\scr X}_\theta$ is a
finite set.

{\rm (2)} If $\mod^{\mu(\delta)}A$ is equivalent to $\mathcal{R'}$,
where $\mathcal{R'}$ is the regular part of $\mod kQ'$ for a
(possibly disconnected) tame quiver $Q'$, then ${\scr X}_\theta$ is
an infinite set.
\end{thm}

\begin{proof} (1) Let ${\cal X}_1$ (resp. ${\cal X}_2$) be the set of isomorphism classes of indecomposable semistable $A$-modules
$M$ (resp.\,with $\mu(M)\neq\mu(\delta)$). According to Theorem {\rm \ref{subcat-equiv}(1)}, we need to
consider the following two cases.

\medskip

{\bf Case 1}. $Q'$ is a Dynkin quiver.

We first show that ${\cal X}_1\cap\mathcal{P}$ is finite. Indeed, by
Corollary \ref{classification of slope}, each module in a
homogeneous tube is not semistable. Hence, there exists
$P\in\mathcal{P}$ such that $\mu(P)>\mu(\delta)$. By the discussion
right above Lemma \ref{existence of homomorphism}, we get that
$\mu(\delta)<\mu(P_i)$ for some $1\leq i\leq N$. Choose $1\leq s\leq
N$ satisfying $\mu(P_j)\leq \mu(P_s)$ for any $1\leq j\leq N$. By
Lemma \ref{dimension vector comparison}, $P_s$ is semistable. Hence,
by Lemma \ref{maximal slope}, there are only finitely many
semistable modules in $\mathcal{P}$.

Next we show that ${\cal X}_1\cap\mathcal{R}$ is finite. By Corollary
\ref{classification of slope}, $A$-modules with dimension vector $m\delta$ are
not semistable. So we only need to consider the non-homogeneous
tubes. Let ${\mathcal T}$ be a non-homogeneous tube of rank $r$. As
before, for $1\leq i\leq r$ and $j\geq 1$, let $E_{i,j}$ be the
indecomposable module in ${\mathcal T}$ with quasi-length $j$ and
quasi-socle $E_i$. Then
$$\bfdim E_{i,j}=\bfdim E_{i,j_0}+n\delta\;\text{ with $j=j_0+nr,\,
0\leq j_0< r$, and $n\geq 1$.}$$
If $\mu(E_{i,j_0})=\mu(\delta)$, then there exists a submodule $N$
of $E_{i,\delta}$ satisfying $\mu(N)>\mu(E_{i,\delta})=\mu(E_{i,j})$
since $E_{i,\delta}$ is not semistable. Hence, $E_{i,j}$ is not
semistable because $N$ is also a submodule of $E_{i,j}$.

If $\mu(E_{i,j_0})\neq \mu(\delta)$, say $\mu(E_{i,j_0}) <
\mu(\delta)$, then $\mu(E_{i,j_0}) < \mu(E_{i,j}) <
\mu(\delta)=\mu(E_{i,\delta})$, which implies that $E_{i,j}$ is not
semistable since $E_{i,\delta}$ is a submodule of $E_{i,j}$.

Consequently, each module in ${\mathcal T}$ with quasi-length $\geq
r$ is not semistable. Therefore, there are only finitely many
semistable modules in $\mathcal{R}$.

By Lemma \ref{equivalent definition} and an argument similar to the
proof for the case of $\mathcal{P}$, we get that
${\cal X}_1\cap\mathcal{I}$ is finite.

\medskip

$\textbf{Case 2}$.  $Q'$ is a tame quiver.

We first show that ${\cal X}_2 \cap\mathcal{P}$ is finite. We will prove
that for $m\gg 0$ and each projective module $P$, if
$M=\tau^{-m}P\notin \mod^{\mu(\delta)}A$, then $M$ is
not semistable. In fact, suppose $M$ is semistable. Let $L \in
\mod^{\mu(\delta)}A\cap\mathcal{P}$ satisfy $\Hom(L,M)\neq0$. Since
$kQ'\cong \mod^{\mu(\delta)}A$ has an infinite preprojective
component, $\mod^{\mu(\delta)}A \cap \mathcal{P}$ is infinite. By
Proposition \ref{existence of homomorphism}, there exists $N_1
 \in \mod^{\mu(\delta)}A\cap \mathcal{P}$ such that $\Hom(M,N_1)\neq0$.
 Thus,
$$\mu(\delta)=\mu(L)\leq\mu(M)\leq\mu(N_1)=\mu(\delta),$$
which implies $\mu(M)=\mu(\delta)$ and $M \in \mod^{\mu(\delta)}A$, a
contradiction. Therefore, ${\cal X}_2 \cap\mathcal{P}$ is finite.

Next we show that ${\cal X}_2\cap \mathcal{R}$ is finite. Since $A$-modules
with dimension vector $m\delta$ do not lie in ${\cal X}_2$, we only need to
consider the non-homogenous tubes. Let ${\mathcal T}$ be a
non-homogenous tube of rank $r$. As in Case 1, for $1\leq i\leq r$
and $j\geq 1$, $\textbf{dim}E_{i,j}=\textbf{dim}E_{i,j_0}+n\delta$
with $j=j_0+nr,\, 0\leq j_0< r$, and $n\geq 1$. Suppose that $E_{i,j}
\notin \mod^{\mu(\delta)}A$. By Proposition \ref{r3},
$\Hom(E_{i,m\delta},E_{i,j})\neq0$ and
$\Hom(E_{i,j},E_{i,(m+1)\delta})\neq0$.

If $E_{i,m\delta}$ is not semistable, then $E_{i,j}$ is not
semistable by an argument similar to Case 1.

If  $E_{i,m\delta}$ is semistable, then $E_{i,(m+1)\delta}$ is
semistable and
$$\mu(\delta)=\mu(E_{i,m\delta})\leq\mu(E_{i,j})\leq\mu(E_{i,(m+1)\delta})=\mu(\delta),$$
which implies that $\mu(E_{i,j})=\mu(\delta)$. Since $E_{i,j}
\notin$ $\mod^{\mu(\delta)}A$, $E_{i,j}$ is not semistable.

In conclusion, the modules in ${\mathcal T}$ with quasi-length $\geq
r$ do not belong to ${\cal X}_2$. Hence, there are only finitely many
indecomposable modules in ${\cal X}_2\cap\mathcal{R}$.

Similarly, we get that ${\cal X}_2 \cap \mathcal{I}$ is finite.

\medskip

(2) We will construct a family of semistable modules $\{ P_{j_i}^{i}
\}_{i\in \mathbb{N}} \in \mathcal{P}$ satisfying
$$\mu( P_{j_i}^{i}) <\mu( P_{j_{i+1}}^{i+1}) < \mu(\delta).$$
 By Corollary \ref{classification of slope}, there exists an indecomposable
semistable module $M$ with $\bfdim M=\delta$. Choose $1\leq j_0\leq
N$ satisfying $\mu(P_j)\leq \mu(P_{j_0})$ for all $1\leq j\leq N$.
Write $P_{j_0}^0$ for $P_{j_0}.$ By Lemma \ref{dimension vector
comparison}, $P_{j_0}^{0}$ is semistable. Since
$\Hom(P_{j_0}^{0},M)\neq 0$, we have $\mu(P_{j_0}^{0})\leq \mu(M)=
\mu(\delta)$. Since $P_{j_0}^{0} \notin \mod^{\mu(\delta)}A$,
$\mu(P_{j_0}^{0})< \mu(M)= \mu(\delta)$. Let
$P_{1}^{1},P_{2}^{1},\ldots,P_{N}^{1}$ be the preprojective modules
in $\mathcal{P}$ with $\bf{dim}$$P_{i}^{1}=\alpha_i+\delta$ for
$1\leq i\leq N$, where $\alpha_1, \ldots, \alpha_N$ are defined
right above Lemma \ref{existence of homomorphism}. Let $1\leq
j_1\leq N$ satisfy $\mu(P_{j}^{1})\leq \mu(P_{j_1}^{1})$ for all
$1\leq j\leq N$. Since
$$\mu(P_{j_1}^{1})\geq \mu(P_{j_0}^{1})>
\mu(P_{j_0}^{0})\geq \mu(P_i), 1\leq i\leq N,$$
 we have by Lemma \ref{dimension vector comparison} that $P_{j_1}^{1}$ is semistable.
Repeating the above process, we finally get a family of semistable modules
$\{P_{j_i}^{i} \}_{i\in \mathbb{N}} \in \mathcal{P}$ satisfying
$\mu(P_{j_i}^{i}) <\mu( P_{j_{i+1}}^{i+1} )< \mu(\delta)$. Hence,
$\{\mu(P_{j_i}^{i}) \}_{i\in \mathbb{N}}\subseteq {\scr X}_\theta$ and
${\scr X}_\theta$ is infinite.

\end{proof}

\begin{example}  Let $Q$ be the tame quiver of type
$\widetilde A_3$ with $A=kQ$:
$$
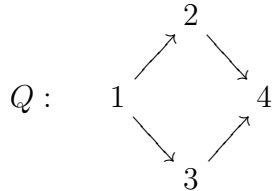
\xymatrix@=0.5cm{  &&2\ar[dr]&\\Q:
&1\ar[ur]\ar[dr]&&4\\
&&3\ar[ur]& }
$$
It is known that $\delta=(1,1,1,1)$.

\medskip

(1) Take $\theta=(1, 1, 2, 0)$. An easy calculation shows that
$${\scr X}_\theta=\{0,1/2,2/3,1,2\}.$$
Moreover, $\mu(\delta)=1$ and $\mod^{1}A$ is equivalent to $\mod
k\Gamma$, where $\Gamma$ is a tame quiver of type $\widetilde A_2$.

\medskip

(2) Take $\theta=(1, 2, 3, 2)$. An easy calculation shows that
$${\scr X}_\theta=\{1,2,5/2,3\}.$$
Moreover, $\mu(\delta)=2$ and $\mod^{2}A$ is equivalent to $\mod k
\Gamma$, where $\Gamma$ is a Dynkin quiver of type $A_2$.

\medskip

(3) Take $\theta=(3, 2, 2, 1)$. Then
$$\{(8n+5)/(4n+3)\mid n\geq 0\} \subseteq {\scr X}_\theta.$$
Hence, ${\scr X}_\theta$ is infinite. Moreover, $\mu(\delta)=2$ and $\mod^{2}A$ consists of all the
regular $A$-modules.
\end{example}

\section*{Acknowledgements}
The author is very grateful to Professor Bangming
Deng for his encouragement and careful guidance. We also thank
Shiquan Ruan, Qinghua Chen and Jie Zhang for many helpful
discussions and concerns.

\end{document}